\documentclass[letterpaper, 10 pt, conference]{ieeeconf}  % Comment this line out
\usepackage{amsmath}
\usepackage{amssymb}
\IEEEoverridecommandlockouts                              % This command is only
\usepackage{graphics} % for pdf, bitmapped graphics files
\usepackage{epsfig} % for postscript graphics files
\usepackage{epstopdf}
\usepackage{mathptmx} % assumes new font selection scheme installed
\usepackage{times} % assumes new font selection scheme installed
\usepackage{amsmath} % assumes amsmath package installed
\usepackage{amssymb}  % assumes amsmath package installed
\usepackage[linesnumbered,lined,boxed,commentsnumbered]{algorithm2e}
\usepackage[noend]{algpseudocode}
\usepackage{subfigure}% in preamble
\usepackage{breqn}
\usepackage{bbm}
\usepackage{dsfont}
\usepackage{enumerate}

\newtheorem{defn}{Definition}
\newtheorem{thm}{Theorem}
\newtheorem{prop}{Proposition}
\newtheorem{lem}{Lemma}

\newtheorem{cor}{Corollary}

\newtheorem{ex}{Example}

\newcommand{\mcl}[1]{\mathcal{ #1}}

\newcommand{\mbf}[1]{\mathbf{ #1}}

\newcommand{\R}{\mathbb{R}}

\newcommand{\N}{\mathbb{N}}

\newcommand{\eps}{\varepsilon}

\title{\LARGE \bf
 Relaxing The Hamilton Jacobi Bellman Equation To Construct Inner And Outer Bounds On Reachable Sets
}

\author{Morgan Jones%
	\thanks{M. Jones is with the School for the Engineering of Matter, Transport and Energy, Arizona State University, Tempe, AZ, 85298 USA. e-mail: {\tt \small morgan.c.jones@asu.edu } },
	Matthew M. Peet% <-this % stops a space
	\thanks{M. Peet is with the School for the Engineering of Matter, Transport and Energy, Arizona State University, Tempe, AZ, 85298 USA. e-mail: {\tt \small mpeet@asu.edu } }
}

\begin{document}

		\maketitle
		\thispagestyle{plain}
		\pagestyle{plain}
\begin{abstract}
We consider the problem of overbounding and underbounding both the backward and forward reachable set for a given polynomial vector field, nonlinear in both state and input, with a given semialgebriac set of initial conditions and with inputs constrained pointwise to lie in a semialgebraic set. Specifically, we represent the forward reachable set using the ``value function'' which gives the optimal cost to go of an optimal control problems and if smooth satisfies the Hamilton-Jacobi-Bellman PDE. We then show that there exist polynomial upper and lower bounds to this value function and furthermore, these polynomial ``sub-value'' and ``super-value'' functions provide provable upper and lower bounds to the forward reachable set. Finally, by minimizing the distance between these ``sub-value'' and ``super-value'' functions in the $L_1$-norm, we are able to construct inner and outer bounds for the reachable set and show numerically on several examples that for relatively small degree, the Hausdorff distance between these bounds is negligible.
\end{abstract}

\section{Introduction}

The reachable set of an ODE is the set of coordinates that can be reached by the solution map, defined in Assumption \ref{ass}, at some fixed time and starting in some set of initial conditions. The computation of reachable sets is important for certifying solution maps remain in ``safety regions''; regions of the state space that are deemed to have low risks of system failure. Historic examples of solution maps transitioning outside ``safe regions'' include: two of the four reaction wheels on the Kepler Space telescope failing, analyzed in \cite{kampmeier2018reaction}; and the disturbing lateral vibrations of the Millennium footbridge over the River Thames in London on opening day, analyzed in \cite{chen2018study} and \cite{eckhardt2007modeling}.

In this paper we show the reachable set of an ODE, subject to pointwise bounded inputs, is the sublevel set of the ``value function" (optimal cost to go function) associated with a one player optimal control problem. This result can be thought of as the analogous result to \cite{mitchell2005time}; where it was shown the reachable set of an ODE, subjected to two sets of adversarially opposed input parameters, is the sublevel set of the ``value function" associated with a two player optimal control problem.

It is known that if the ``value function" of a one player optimal control problem is smooth then it satisfies the Hamilton Jacobi Bellman (HJB) Partial Differential Equation (PDE) \cite{bertsekas2005dynamic}. In this paper we show that relaxing the HJB PDE to a dissipation inequality allows for the construction of upper and lower bounds of the "value function''; we call super-value and sub-value functions respectively. We futhermore give sufficient conditions for the existence of polynomial super-value and sub-value functions. Moreover, it is shown that the sublevel set of sub-value and super-value functions construct provable upper and under bounds of reachable sets respectively.

The HJB PDE may not always have a solution in the classical sense. A generalized solution concept, called the viscosity solution, was developed in \cite{crandall1992user}. Discretization methods, such as those in \cite{wang2017adaptive} \cite{wan2019alternating}, are typically used to approximate the viscosity solution. However, such methods cannot guarantee that the approximate viscosity solution is an upper or lower bound to the true "value function''. Alternatively, we propose a Sum-of-Squares (SOS) optimization problem that is solved by the polynomial sub-value and super-value functions with minimum $L_1$ distance. %It is shown for relatively small degrees the SOS optimization problem is solved by a sub-value function that is tight to the true "value function'' under the $L_1$ metric.

Our approach to finding sub- and super-solutions to the HJB PDE is similar to \cite{summers2013approximate} and \cite{leong2016linearly}. In \cite{summers2013approximate} SOS was used to find a sub-value function for optimal control problems with discrete-time dynamics; whereas we consider continuous-time dynamics. In \cite{leong2016linearly} SOS was used to find sub-value and super-value functions for optimal control problems with quadratic costs and continuous-time synamics governed by ODE's affine in the input variable. Our approach allows us to construct sub-value and super-value functions for more general optimal control problems with polynomial costs and continuous time varying processes governed by ODE's nonlinear in the input variable. Moreover, we give sufficient conditions on the existence of polynomial sub-value and super-value functions and show how these functions can be used for reachable set estimation.

We numerically demonstrate that solving our proposed SOS optimization problem can give tight approximations of reachable sets. Unlike alternative approaches to reachable set analysis, \cite{mitchell2005time} \cite{greenstreet1999reachability} \cite{maidens2015reachability}, our reachable set approximations can be proved to overbound or underbound the reachable set.

An alternative approach to reachable set approximation is found in \cite{summers2013quantitative} \cite{yin2018reachability} \cite{yin2018finite} \cite{jones2019using} where dissipation like inequalities are solved using SOS programing to find a function whose sublevel set contains the reachable set. It is shown in this paper such dissipation inequalities are actually relaxations of the HJB PDE and thus solved by sub-value functions. In this context, our sufficient conditions for the existence of polynomial sub-value functions for optimal control problems can be viewed as feasibility conditions for the SOS optimization problems found \cite{summers2013quantitative} \cite{yin2018reachability} \cite{yin2018finite} \cite{jones2019using}.

The paper is organized as follows. Background material on ODE's is given in Section \ref{sec: Backround ODE}. In Section \ref{Section: finite time} optimal control theory is presented. In Section \ref{sec: reachable sets descibred by value} we construct an optimal control problem with value function that can characterize the reachable set exactly. In Section \ref{Section: Dissipation inequalities for sub-value functions} we show how relaxing the HJB PDE allows us to derive dissipation inequalities that are solved by sub-value and super-vale functions. In Section \ref{sec: SOS to solve optimal control} an SOS optimization is proposed that minimizes the $L_1$ norm of the distance between the sub-value and super-value function. The conclusion is given in Section \ref{sec: conclusion}.

\section{Notation} \label{sec: Notation}
We denote a ball with radius $R>0$ centered at the origin by $B_R=\{x \in \R^n: x^Tx< R^2\}$. For $x \in \R^n$ we denote $||x||_\infty= \max_{1 \le i \le n}|x_i|$. For short hand we denote the partial derivative $D^\alpha f(x):= \Pi_{i=1}^{n} \frac{\partial^{\alpha_i} f}{\partial x_i^{\alpha_i}} (x)$ for $\alpha \in \N^n$. Let $C(\Omega)$ be the Banach space of scalar continuous functions with domain $\Omega \subset \R^n$. For $f \in C(\Omega)$ we define the norms $||f||_{\infty, \Omega}:= \sup_{x \in \Omega}||f(x)||_{\infty}$ and $||f||_{1, \Omega}=\int_{\Omega}|f(x)|dx$. We denote the set of differentiable functions by $C^i(\Omega):=\{f \in C(\omega): D^\alpha f \in C(\Omega) \text{ } \forall \alpha \in \N^n \text{ such that } \sum_{j=1}^{n} \alpha_j \le i\}$. For $V \in C^1(\R^n \times \R)$ we denote $\nabla_x V:= (\frac{\partial V}{\partial x_1},....,\frac{\partial V}{\partial x_1})$ and $\nabla_t V=\frac{\partial V}{\partial x_{n+1}} $. For $d \in \N$ and $x \in \R^n$ we denote $z_d(x)$ to be the vector of monomial basis in $n$-dimensions with maximum degree $d \in \mathbb{N}$. We denote the space of scalar valued polynomials $p: \Omega \to \R$ with degree at most $d \in \N$ by $\mcl{P}_d[\Omega]$. We say $p \in \mcl{P}_d[\R^n]$ is Sum-of-Squares (SOS) if there exists $p_i \in \mcl{P}_d[\R^n]$ such that $p(x) = \sum_{i=1}^{k} (p_i(x))^2$. We denote $\sum_{SOS}$ to be the set of SOS polynomials.

\section{Background: Differential Equations} \label{sec: Backround ODE}
We consider nonlinear Ordinary Differential Equations (ODE's) of the form
\begin{equation} \label{eqn: ODE}
\dot{x}(t) = f(x(t), \mbf u(t)), \quad \mbf u(t) \in Y, \quad x(0)\in X_0\in \R^n,
\end{equation}
where $f:\R^n \times \R^{m} \to \R^n$; $u: \R \to \R^{m}$ is the input; and $Y \subset \R^{m} $ and $X_0 \subset \R^{n} $ are compact sets representing constraints on the inputs and initial conditions.

%Throughout this paper we will assume the existence and uniqueness of solution maps, defined next, associated with ODE's.
%\begin{defn} \label{defn: soln map}
%	We say $\phi: \R^n \times \R  \times L_2^{m}[0,T]  \to \R^n$ is a solution map for \eqref{eqn: ODE} if $\frac{\partial \phi(x,t,u)}{\partial t}= f(\phi(x,t,u),u(t))$ and $\phi(x,0,u)=x$.
%\end{defn}

%To define the solution map, for a given $T>0$ which will be clear from context, we define the set of pointwise-admissible input signals as
To define the solution map we define the set of pointwise-admissible input signals as
\[
U_Y:=\{\mbf u: \R \to \R^m: \mbf u(t) \in Y \text{ for all } t \in (-\infty,\infty) \}.
\]
  For a given set of admissible inputs, we constrain $f$, in the following definition, to admit a continuously-differentiable solution map.
\begin{defn}[Constraint on Admissibility of $f$] \label{ass} For given $Y \subset \R^m$ we say $f \in \mcl F_Y$ if
	\begin{enumerate}
		\item $f_i \in C^1(\R^n)$ for all $i \in \{1,...,n\}$.
		\item For any $T>0$, there exists a function $h: \R^n \times \R  \times U_Y  \to \R^n$, where for any $\mbf u \in U_Y$ we have $h_i(\cdot,\cdot, \mbf u) \in  C^1(\R^n \times [-T,T])$ for all $i \in \{1,...,n\}$, and
		\begin{align} \label{eqn: defining proprty of solution map}
		&\frac{\partial h(x,t,\mbf u)}{\partial t}= f(h(x,t,\mbf u),\mbf u(t))\\ \nonumber
		&h(x,0,\mbf u)=x,
		\end{align}
		for all $ x \in \R^n$, $t \in [-T,T]$ and $\mbf u \in U_Y$.
		\item The function $h$ that satisfies \eqref{eqn: defining proprty of solution map} is unique.
	\end{enumerate}
Since for each $f \in \mcl F_Y$ the associated function that satisfies \eqref{eqn: defining proprty of solution map} is unique we will denote this function by $\phi_f$ throughout the paper.
\end{defn}
%We now prove useful properties of the solution map.
\begin{lem} \label{lem: soln map properties}
	Let $Y \subset \R^m$ be a compact set, $f, -f\in \mcl F_Y$ and $T \in \R^+$.
	\begin{enumerate}[(A)]%[label=(\Alph*)]
		\item For $\mbf u \in U_Y$ define $\hat{\mbf u} (t)=\mbf u(-t)$, then $\forall x \in \R^n, t \in [-T,T]$	
		\begin{align} \label{eqn: reverse time solution map}
		& \phi_{-f}(x,-t,\hat{\mbf u})= \phi_f(x,t,\mbf u). \end{align}
		
		\item For $s \in [-T,T]$ and $\mbf u \in U_Y$ define $\tilde{ \mbf u}_s(t)=\mbf u(t +s)$, then $\forall x \in \R^n, t \in [-T-s,T-s] \cap [-T,T]$
		  \begin{align}
			& \phi_f(x,t+s,\mbf u)= \phi_f(\phi_f(x,s,\mbf u),t,\tilde{\mbf u}_s) \label{eqn: soln map semi group proerty}
		\end{align}
		
%		\item For $t \in[-T,T]$ and $\mbf u \in U_Y$ suppose $ \mbf w \in U_Y$ is such that $\mbf w(s)= \mbf u(s)$ for all $s \in [-T,t]$, then $\forall x \in \R^n$
%		  \begin{align}
%			& \phi_{f}(x,t,{\mbf u})= \phi_f(x,t,\mbf w). \label{eqn: same controller}
%		\end{align}
	\end{enumerate}	
	
\end{lem}
\begin{proof}
	\underline{\textbf{Proving \eqref{eqn: reverse time solution map} in Statement (A):}} As $-f \in \mcl F_Y$ we have for all $x \in \R^n$, $t \in [-T,T]$, and $\mbf u \in U_Y$
	\begin{align} \label{defining property of phi_-f}
	\frac{\partial \phi_{-f}(x,t,\mbf u)}{\partial t}= -f(\phi_{-f}(x,t,\mbf u), \mbf u(t)) \text{ and } \phi_{-f}(x,0,\mbf u)=x.
	\end{align}
	Now, letting $h_1(x,t, \mbf u)= \phi_{-f}(x,-t, \hat{\mbf u})$, for $x \in \R^n$, $t \in [-T,T]$, and $\mbf u \in U_Y$ the following holds
	\begin{align*}
	\frac{\partial h_1(x,t,\mbf u)}{\partial t} & = \frac{\partial \phi_{-f}(x,-t,\hat{\mbf u})}{\partial t}=-\frac{\partial \phi_{-f}(x,s,\hat{\mbf u})}{\partial s}\\
	& = f(\phi_{-f}(x,s,\hat{\mbf u}),\hat{\mbf u}(s)) = f(\phi_{-f}(x,-t,\hat{\mbf u}),\mbf u(t))\\
	& = f(h_1(x,t,{\mbf u}),\mbf u(t)),
	\end{align*}
	where to get the second equality we use $s=-t$, so $ds=-dt$; to get the third equality \eqref{defining property of phi_-f} was used; to get the fourth equality the substitution $s=-t$ was again applied, noting $\hat{ \mbf u}(-t)=\mbf u(t)$.
	Moreover, as $h_1(x,0,\mbf u)=\phi_{-f}(x,0,\hat{\mbf u})=x$, by \eqref{defining property of phi_-f}, it follows $h_1$ satisfies \eqref{eqn: defining proprty of solution map} and therefore, due to the uniqueness of $\phi_f$, \eqref{eqn: reverse time solution map} must follow.
	
	\underline{\textbf{Proving \eqref{eqn: soln map semi group proerty} in Statement (B):}} For fixed $s \in [-T,T]$ let us consider the following function
	\begin{align*}
	h_2^s(x,t, \mbf u) := \begin{cases}
	\phi_f(x,t, \mbf u) \text{ for } -T \le t \le s\\
	\phi_f(\phi_f(x,s, \mbf u),t-s,\tilde{\mbf u}_s) \text{ for } t \in (s,T].
	\end{cases}
	\end{align*}
	We prove \eqref{eqn: soln map semi group proerty} by showing $h_2$ satisfies \eqref{eqn: defining proprty of solution map} and using the uniqueness properties of $\phi_f$. Firstly it is clear $h_2^s(x,0, \mbf u)= \phi_f(x,0,\mbf u)=x$, and $h_2^s$ satisfies \eqref{eqn: defining proprty of solution map} for all $x \in \R^n$, $t \in [-T,s]$ and $\mbf u \in U_Y$. Now for all $x \in \R^n$, $t \in (s,T]$ and $\mbf u \in U_Y$
	\begin{align*}
	& \frac{\partial h_2^s(x,t,\mbf u)}{\partial t}  = \frac{\partial \phi_{f}(\phi_f(x,s,\mbf u),t-s,\tilde{\mbf u}_s)}{\partial t}  \\
	& = \frac{\partial \phi_{f}(\phi_f(x,s,\mbf u),k,\tilde{\mbf u}_s)}{\partial k} = f(\phi_{f}(\phi_f(x,s,\mbf u),k,\tilde{\mbf u}_s),\tilde{\mbf u}_s(k)) \\
	&= f(\phi_{f}(\phi_f(x,s,\mbf u),t-s,\tilde{\mbf u}_s),\tilde{\mbf u}_s(t-s)) = f(h_2(x,t,\mbf u), \mbf u (t)),
	\end{align*}
	where the second equality follows from using $k=t-s$ so $dk=dt$; the third equality follows by \eqref{eqn: defining proprty of solution map}; the fourth equality follows from applying $k=t-s$ again; the fifth equality follows as $\tilde{\mbf u}_s(t-s)= \mbf u(t)$.
	
	Thus by the uniqueness of $\phi_f$ it follows $\phi_f(x,t+s, \mbf u)=h_2(x,t+s, \mbf u)= \phi_f(\phi_f(x,s,\mbf u),t,\tilde{\mbf u}_s)$, therefore showing \eqref{eqn: soln map semi group proerty}.
	
%	\underline{\textbf{Proving \eqref{eqn: same controller} in Statement (C):}} For $\mbf u \in U_Y$ and $t \in [-T,T]$ suppose $\mbf w \in U_Y$ is such that $\mbf w(s)= \mbf u(s)$ for all $s \in [-T,t]$. Let us consider the following function
%		\begin{align*}
%	h_3(x,s, \mbf u) := \begin{cases}
%	\phi_f(x,s, \mbf w) \text{ for } 0 \le s \le t\\
%	\phi_f(x,s,\mbf u) \text{ for } s \in (t,T].
%	\end{cases}
%	\end{align*}
%	By a similar argument to the previous part of the proof it can be shown $h_3$ satisfies \eqref{eqn: defining proprty of solution map} and thus by the uniqueness of $\phi_f$ it follows $\phi_f(x,t, \mbf w)=h_3(x,t, \mbf u)=\phi_f(x,t,\mbf u)$; proving \eqref{eqn: same controller}.
\end{proof}

For a given $X_0 \subset \R^n$, $Y\subset \R^n$ and $f \in \mcl F_Y$, we next define the forward reachable set as follows.

\begin{defn} \label{defn: reachbale set}
	For $X_0 \subset \R^n$, $Y \subset \R^m$, $f\in \mcl F_Y$ and $S \subset \R^+$, let
	{ \begin{align*}
		FR_f(X_0,Y,S):= \{y \in \R^n \;:& \;\exists x \in X_0, \mbf u \in U_Y, \text{and } t \in S \\
		& \text{such that } \phi_f(x,t,\mbf u)=y  \}.
		\end{align*} } \normalsize
%	Moreover, for $S \subseteq \R^+$ we define the forward reachable tube as
%		{ \begin{align*}
%		FR(X_0,f,S,Y):= \cup_{s \in S} FR(X_0,f,s,Y) .
%		\end{align*} } \normalsize
\end{defn}

In following sections, $S$ is of the form either $\{T\}$ or $[0,T]$.

\section{Finite Time Optimal Control Problems} \label{Section: finite time}

An optimal control problem with finite time horizon is a tuple $\{c,g,f,X_0,Y,T\}$ where   $c: \R^n \times \R \times \R \to \R$ is the running cost; $g: \R^n \to \R$ is the terminal cost; $f \in \mcl F_Y$; $X_0 \subset \R^n$ is the set of initial conditions; $Y \subset \R^m$ is a compact input set; and $T$ is the final time. For each optimal control problem we can next define the value function that intuitively describes the optimal "cost to go''.

\begin{defn}
	For given $X_0 \subset \R^n$; $Y \subset \R^m$; $T>0$;  $c: \R^n \times \R \times \R \to \R$; $g: \R^n \to \R$; $f \in \mcl F_Y$ we say $V^*: \R^n \times \R \to \R$ is a value function of the tuple $\{c,g,f,X_0,Y,T\}$ if for $(x,t) \in (FR_f(X_0,Y,\{t\}),t)$, where $t \in [0,T]$, the following holds
		\begin{align}
	&V^*(x,t)=\label{opt: optimal control general}\\
	&\inf_{\mbf u \in U_Y} \left\{ \int_{t}^{T} c(\phi_f(x,s-t,\mbf u),\mbf u(s),s) ds + g(\phi_f(x,T-t,\mbf u)) \right\}.\notag
	\end{align}
\end{defn}

A sufficient condition for $V^*$ to be a value function for the tuple $\{c,g,f,X_0,Y,T\}$ is for $V^*$ to satisfy the Hamilton Jacobi Bellman (HJB) PDE. 

\begin{prop} \label{prop: HJB eqn optimal}
		For given $X_0 \subset \R^n$, $Y \subset \R^m$, $g: \R^n \to \R$, $c:\R^n \times \R^m \times \R \to \R$, $f \in \mcl F_Y$, $T>0$, suppose there exists a differentiable function $V \in C^1(\R^n \times \R)$ such that the following holds for $(x,t) \in FR_f(X_0,Y,[0,T]) \times [0,T]$
		\begin{align}  \nonumber
		&\nabla_t V(x,t) + \inf_{u \in Y} \left\{ c(x,u,t) + \nabla_x V(x,t)^T f(x,u) \right\} = 0 \\  \label{eqn: general HJB PDE}
		& V(x,T)= g(x).
		\end{align}
 Then $V$ is the value function of the optimal control problem $\{c,g,f,X_0,Y,T\}$. %and thus
%		\small{	\begin{equation*}
%			V(x,t)= \inf_{u \in U_Y} \left\{ \int_{t}^{T} c(\phi_f(x,s,u),u(s),s) ds + g(\phi_f(x,T-t,u)) \right\}.
%			\end{equation*} } \normalsize
\end{prop}
\begin{proof}
	Follows by Proposition 3.2.1 from \cite{bertsekas2005dynamic} where the domain of the value function is restricted to $(x,t) \in (FR_f(X_0,Y,\{t\}),t)$.
	\end{proof}

\begin{defn} \label{defn: sub and super value functions}
	We say the function $J: \R^n \times \R \to \R$ is a sub-value function to the finite time horizon optimal control problem $\{c,g,f,X_0,Y,T\}$ if we have
	\begin{align*}
	J(x,t) \le V^*(x,t) \quad \forall t \in [0,T] \text{ and } x \in FR_f(X_0,Y,\{t \}),
	\end{align*}
	where $V^*$ is the value function of $\{c,g,f,X_0,Y,T\}$. Moreover if instead $J$ satisfies
		\begin{align*}
	J(x,t) \ge V^*(x,t) \quad \forall t \in [0,T] \text{ and } x \in FR_f(X_0,Y,\{t \}),
	\end{align*}
	 we say $J$ is a super-value function to $\{c,g,f,X_0,Y,T\}$.
\end{defn}

%\subsection{Infinite Time Horizon Optimal Control}

\section{How Sublevel Sets Of Value Functions Can Describe Reachable Sets} \label{sec: reachable sets descibred by value}
In this section we construct a finite time horizon optimal control problem with associate value function whose sublevel sets can construct the reachable set of a system. We then show how the sublevel sets of the sub-value and super-value functions over- and under-bound the reachable set.

%First, we use our definition of solution map to define reverse-time trajectories. Specifically, it is relatively easy to show that $\phi_f$ satisfies the following properties.

Analogous to Definition \ref{defn: reachbale set} we now define the backward reachable set and show how it is related to the forward reachable set in Lemma \ref{lem: backward and forward reach sets}.
\begin{defn} \label{defn: backward reachbale set}
	For $X_0 \subset \R^n$, $Y \subset \R^m$, $f\in \mcl F_Y$ and $S \subset \R^+$, let
	{ \begin{align*}
		BR_f(X_0,Y,S):= \{y \in \R^n \;:& \;\exists x \in X_0, \mbf u \in U_Y, \text{and } t \in S \\
		& \text{such that } \phi_f(y,t,\mbf u)=x  \}.
		\end{align*} } \normalsize
	%	Moreover, for $S \subseteq \R^+$ we define the forward reachable tube as
	%		{ \begin{align*}
	%		FR(X_0,f,S,Y):= \cup_{s \in S} FR(X_0,f,s,Y) .
	%		\end{align*} } \normalsize
\end{defn}
In the next Lemma we give a relationship between the backward reachable set and forward reachable set. This relationship shows finding the set $FR_{f}(X_0,Y,\{T\})$ is equivalent to finding the set $BR_{-f}(X_0,Y,\{T\})$. Therefore for the rest of this paper we concentrate on developing methods to bound the backward reachable set. However, for numerical implementation we will change the sign of the vector field to allow for the calculation of forward reachable set bounds.
\begin{lem} \label{lem: backward and forward reach sets}
	Suppose $Y \subset \R^m$, $f\in \mcl F_Y$  is such that $-f\in \mcl F_Y$, and $T \in \R^+$. Then $FR_{-f}(X_0,Y,\{T\})=BR_f(X_0,Y,\{T\}).$
%	\begin{align*}
%	FR_{-f}(X_0,T,Y)=BR_f(X_0,T,Y).
%	\end{align*}
\end{lem}
\begin{proof}
We first show $FR_{-f}(X_0,Y,\{T\}) \subseteq BR_f(X_0,Y,\{T\})$. For $y \in FR_{-f}(X_0,Y,\{T\})$ there exists $x \in X_0$ and $ \mbf u \in U_Y$ such that
\begin{align} \label{one}
\phi_{-f}(x,T, \mbf u)=y.
\end{align}
If we denote $\hat{\mbf u}(t)=\mbf u(-t)$ and $\tilde{\mbf u}(t)= \hat{\mbf u}(t-T)$, it now follows
\begin{align} \nonumber
\phi_f(y,T,\tilde{\mbf u}) & =\phi_f(\phi_{-f}(x,T,\mbf u), T, \tilde{ \mbf u}) =   \phi_f(\phi_{f}(x,-T,\hat{\mbf u}), T, \tilde{ \mbf u})\\
& = \phi_f(x, T -T, \hat{ \mbf u})=x, \label{two}
\end{align}
where the first equality follows by \eqref{one}, the second equality by \eqref{eqn: reverse time solution map}, and the third equality follows by \eqref{eqn: soln map semi group proerty}. Thus we deduce from \eqref{two} $y \in BR_f(X_0,Y,\{T\})$.

We next show $BR_f(X_0,Y,\{T\}) \subseteq FR_{-f}(X_0,Y,\{T\})$. For $y \in BR_f(X_0,Y,\{T\})$ there exists $x \in X_0$ and $ \mbf u \in U_Y$ such that
\begin{align} \label{three}
\phi_f(y,T,\mbf u)=x.
\end{align}
Let us denote $\mbf w(t) = u(t+T)$, $\hat{\mbf w}(t) = \mbf w(-t)$ then it now follows
\begin{align*}
\phi_{-f}(x,T, \mbf w)& = \phi_{f}(x,-T, \hat{\mbf w})=\phi_f(\phi_f(y,T,\mbf u),-T,\hat{\mbf w})\\
& = \phi_f(y,T-T, \mbf u)=y,
\end{align*}
where the first equality follows \eqref{eqn: reverse time solution map}, the second equality by \eqref{three}, and the third equality by \eqref{eqn: soln map semi group proerty}. Thus we deduce $y \in FR_{-f}(X_0,Y,\{T\})$.
\end{proof}

\begin{thm} \label{thm: HJB to characterize reachable sets}
Given $Y \subset \R^m$, $f \in \mcl F_Y$ and $g: \R^n \to \R$, let $X_0= \{x \in \R^n: g(x) \le 1\}$ and $X \subset \R^n$ be such that $BR_f(X_0,Y,\{T\}) \subseteq X$. Now suppose $V^*: \R^n \times \R \to \R$ is a value function for $\{0,g,f,X,Y,T\}$, then
%$\{x \in \R^n : V^*(x,0) \le 1\} \subseteq B_r$ then
	\begin{equation} \label{forward reach is sublevel set}
	BR_f(X_0,Y,\{T\}) = \{x \in X : V^*(x,0) \le 1\}.
	\end{equation}
\end{thm}
\begin{proof}
%By Lemma \ref{lem: backward and forward reach sets} $FR_{-f}(X_0,Y,\{T\})=BR_f(X_0,Y,\{T\})$ thus to prove \eqref{forward reach is sublevel set} we will prove $BR_f(X_0,Y,\{T\})= \{x \in \R^n : V^*(x,0) \le 1\}$.
	As $V^*$ is a value function to $\{0,g,f,X,Y,T\}$ it follows for all $t \in [0,T]$ and $x \in FR_f(X,Y,\{t\})$
	\begin{align} \label{Terminal value function}
	V^*(x,t)= \inf_{\mbf u \in U_Y} g(\phi_f(x,T-t,\mbf u)).
	\end{align}
%	In particular when $t=0$ \eqref{Terminal value function} only holds when $x \in X$.
	
	For $y_0 \in BR_f(X_0,Y,\{T\}) \subseteq X$ there exists $x_0 \in X_0$ and $\mbf u_0 \in U_Y$ such that $\phi_{f}(y_0,T,\mbf u_0)=x_0$. Thus it follows
	{\small \begin{align*}
	V^*(y_0,0) & = \inf_{ \mbf u \in U_Y} g(\phi_f(y_0,T, \mbf u)) \le g(\phi_f(y_0,T, \mbf u_0))= g(x_0) \le 1,
	\end{align*} }
where the first equality follows as $y_0 \in X$ so \eqref{Terminal value function} holds. Therefore $y_0 \in \{x \in X : V^*(x,0) \le 1\}$. Hence $BR_f(X_0,Y,\{T\}) \subseteq \{x \in X : V^*(x,0) \le 1\}$.
	
	Now suppose $y_0 \in \{x \in X : V^*(x,0) \le 1\}$. Then if $ \mbf u_0: = \arg \inf_{ \mbf u \in U_Y} g(\phi_f(y_0,T, \mbf u))$, let  $x_0:= \phi_f(y_0,T, \mbf u_0)$. It follows
	{\small \begin{align*}
	g(x_0)=g(\phi_f(y_0,T,\mbf u_0))=\inf_{\mbf u \in U_Y}g(\phi_f(y_0,T,\mbf u))=V^*(y_0,0) \le 1,
	\end{align*} }
	where the third equality follows because $y_0 \in X$ so \eqref{Terminal value function} holds. Hence $x_0 \in X_0$. Therefore $y_0 \in BR_f(X_0,Y,\{T\})$. Thus $\{x \in X : V^*(x,0) \le 1\} \subseteq BR_f(X_0,Y,\{T\})$.
\end{proof}

%Moreover by a similar argument to Theorem \ref{thm: HJB to characterize reachable sets} an optimal control problem can be constructed with value function that has a sublevel set that can construct backward reachable sets.

We next show how sub-value and super-value functions, defined in Definition \ref{defn: sub and super value functions}, can can outer bound and inner bound reachable sets.

\begin{lem} \label{lem: under and over approx of reachable set}
	Given $Y \subset \R^m$, $f \in \mcl F_Y$ and $g: \R^n \to \R$, let $X_0= \{x \in \R^n: g(x) \le 1\}$ and $X \subset \R^n$ be such that $BR_f(X_0,Y,\{T\}) \subseteq X$. Suppose $V_l$ and $V_u$ are sub-value and super-value functions to the optimal control problem $\{0,g,f,X,Y,T\}$. Then
	\begin{align} \label{set containments statement}
	 \{x \in X : V_u(x,0) \le 1\} &  \subseteq BR_{f}(X_0,Y,\{T\})\\ \nonumber
	 BR_{f}(X_0,Y,\{T\}) & \subseteq \{x \in X : V_l(x,0) \le 1\}.
	\end{align}
%	where $X_0=\{x \in \R^n: g(x) \le 1\}$.
\end{lem}
\begin{proof}
	Since $V_l$ and $V_u$ are sub-value and super-value functions to the optimal control problem $\{0,g,f,X,Y,T\}$ it follows $\forall t \in [0,T] \text{ and } x \in FR_f(X,Y,\{t\})$
	\begin{align} \label{sub and super valu}
	V_l(x,t) \le V^*(x,t) \le V_u(x,t),
	\end{align}
	where $V^*$ is the value function to $\{0,g,f,X,Y,T\}$.
	
	By \eqref{sub and super valu} it follows
		\begin{align} \label{set containments}
	\{x \in X : V_u(x,0) \le 1\} & \subseteq \{x \in X : V^*(x,0) \le 1\}\\ \nonumber
	\{x \in X : V^*(x,0) \le 1\} & \subseteq \{x \in X : V_l(x,0) \le 1\}.
	\end{align}

	Moreover by Theorem \ref{thm: HJB to characterize reachable sets} we have
	\begin{align} \label{forward reach is sublevel}
	BR_{f}(X_0,Y,\{T\}) = \{x \in X : V^*(x,0) \le 1\}.
	\end{align}
Thus \eqref{set containments} together with \eqref{forward reach is sublevel} proves the set containments given in \eqref{set containments statement}.
\end{proof}

\section{Dissipation Inequalities For Sub-Value and Super-Value Functions} \label{Section: Dissipation inequalities for sub-value functions}
We now propose dissipation inequalities and show, using a novel proof, that if a differentiable function satisfies such inequalities then it must be a sub-value or super-value function associated with an optimal control problem. The dissipation inequalities are found by relaxing the HJB PDE to an inequality. A similar result is found  in Theorem 3.3, from \cite{crandall1992user}, for a class of PDE's that include the HJB PDE. However in \cite{crandall1992user} a futher property, the candidate sub-value function is less than or equal to the candidate super-value function on the boundary of some compact set, is required to hold before such functions can be verified as sub-value and super-value functions.

%It is similarily proved in Theorem 3.3, found in \cite{crandall1992user}, that ``sub-solutions" and ``super-solutoins" of a class of PDE's, including the HJB PDE, satisfy the relaxation of the PDE to an inequality. However, the result from \cite{crandall1992user} requires the 

%We then give conditions that ensure the existence of a polynomial function that satisfies such dissipation inequalities; thus proving the existence of polynomial sub-value and super-value functions for finite time horizon optimal control problems.

\begin{prop} \label{prop: diss ineq implies lower soln}
	For given $T>0$, compact $Y \subset \R^m$, $g \in C^1(\R^n)$, $c \in C^1(\R^n \times \R^m \times \R)$, $f \in \mcl F_Y$. Suppose $X_c \subseteq \R^n$ is such that $FR_f(X_0,Y,[0,T]) \subseteq X_c$ and $J \in C^1(\R^n \times \R)$ satisfies the following $\forall x \in X_c, u \in Y, t\in [0,T]$
	\begin{align} \label{ineq: diss ineq for sub sol of HJB}
	& \nabla_t J(x,t) + c(x,u,t) + \nabla_x J(x,t)^T f(x,u) \ge 0 \\ \label{ineq: BC}
	& J(x,T) \le g(x).
	\end{align}
	 Then $J$ is a sub-value function to the optimal control problem $\{c,g,f,X_0,Y,T\}$.
	
	Alternatively if $J$ satisfies the following $\forall x \in X_c, u \in Y, t\in [0,T]$
		\begin{align} \label{ineq: super diss ineq for sub sol of HJB}
	& \nabla_t J(x,t) + c(x,u,t) + \nabla_x J(x,t)^T f(x,u) \le 0 \\ \label{ineq: super BC}
	& J(x,T) \ge g(x),
	\end{align}
Then $J$ is a super-value function to $\{c,g,f,X_0,Y,T\}$.
\end{prop}
\begin{proof}
%	It can be shown Inequality \eqref{ineq: diss ineq for sub sol of HJB} holds if and only if the following inequality holds
%		\begin{align} \label{ineq: equiv}
%	& \frac{ \partial J}{\partial t}(x,T-t) - c(x,u,T-t) - \nabla J(x,T-t)^T f(x,u) \le 0
%	\end{align}
	Let us denote the left hand side of Inequality \eqref{ineq: diss ineq for sub sol of HJB} by,
	\begin{align*}
	L(x,t,u)= \nabla_t J(x,t) + c(x,u,t) + \nabla_x J(x,t)^T f(x,u).
	\end{align*}
	As $Y$ is compact and the functions $c$ and $f$ are both differentiable we may define $\tilde{L}(x,t):=\inf_{u \in Y} L(x,t,u)$. Moreover we deduce from Inequality \eqref{ineq: diss ineq for sub sol of HJB} that $\tilde{L}(x,t) \ge 0$ for all $x \in X_c$ and $t\in[0,T]$. Now from the construction of the function $\tilde{L}$ it is clear $J$ satisfies the following equation for any $x \in X_c, u \in Y, t\in [0,T]$
	\begin{align} \label{1}
& \nabla_t J (x,t)+ \inf_{u \in Y} \left\{  c(x,u,t) -\tilde{L}(x,t) + \nabla_x J(x,t)^T f(x,u) \right\} =0.
\end{align}
%	We now formulate a HJB PDE by making the change of coordinates $s=T-t$ and rearranging \eqref{1}:
%		\begin{align} \nonumber
%	& \frac{ \partial \hat{J}}{\partial s}(x,s)+ \inf_{u \in Y} \{  {c}(x,u,s) -\hat{L}(x,s) + \nabla \hat{J}(x,s)^T f(x,u) \} =0\\ \label{2}	
%	& \hspace{4cm} \forall x \in X_c, u \in Y, s\in [0,T],
%	\end{align}
%	where the property $\sup_{u \in Y} F(u)= - \inf_{u \in Y} -F(u)$, for some function $F: \R^m \to \R$, was used.
	If we consider the optimal control problem $\{\tilde{c},\tilde{g},f,X_0,Y,T\}$,  where $\tilde{c}(x,u,t)=c(x,u,t)-\tilde{L}(x,t)$ and $\tilde{g}(x)=J(x,T)$, as \eqref{1} holds $\forall x \in FR_f(X_0,Y,[0,T]) \subseteq X_c$ and $t \in [0,T]$ it follows by Proposition \ref{prop: HJB eqn optimal} $J$ is a value function for $\{\tilde{c},\tilde{g},f,X_0,Y,T\}$. It now follows for any $t \in [0,T]$ and $x \in FR_f(X_0,Y,\{t\})$ we have

{\small	\begin{align} \nonumber
	&	{J}(x,t) = \inf_{\mbf u \in U_Y} \bigg\{ \int_{t}^{T}  \tilde{c}(\phi_{f}(x,s-t, \mbf u), \mbf u(s),s) ds+ \tilde{g}(\phi_{f}(x,T-t, \mbf u) \bigg\} \\ \nonumber
	& =\inf_{\mbf u \in U_Y} \bigg\{ \int_{t}^{T}  c(\phi_{f}(x,s-t, \mbf u), \mbf u(s),s) -\tilde{L}(\phi_{f}(x,s-t, \mbf u),s) ds \\ \nonumber
		& \hspace{2cm}+ {J}((\phi_{f}(x,T-t, \mbf u),T) \bigg\} \\ \nonumber
		& \le \inf_{\mbf u \in U_Y} \left\{ \int_{t}^{T} c(\phi_{f}(x,s-t, \mbf u), \mbf u(s),s) ds + g(\phi_{f}(x,T-t, \mbf u) \right\} \\ \label{proof J is sub value}
		& = V^*(x,t),
	\end{align} }
	where $V^*$ is a value function of $\{c,g,f,X_0,Y,T\}$, and the inequality follows from the fact $\tilde{L}(x,t) \ge 0$ for all $x \in X_c$ and $t \in [0,T]$, thus implying $\tilde{L}(\phi_f(x,s-t,\mbf u),s) \ge 0$ for all $x \in FR_f(x_0,Y,\{t\})$ and $s \in [t,T]$; and the fact $J(x,T) \le g(x)$ for all $x \in X_c$, thus implying ${J}((\phi_f(x,T-t, \mbf u),T) \le g(\phi_f(x,T-t, \mbf u))$ for any $x \in FR_f(x_0,Y,\{t\})$. Therefore it is clear from \eqref{proof J is sub value} that $J$ is a sub-value function to $\{c,g,f,X_0,Y,T\}$.
	
	We now prove if the Inequalities \eqref{ineq: super diss ineq for sub sol of HJB} and \eqref{ineq: super BC} hold then $J$ is a super-value function to $\{c,g,f,X_0,Y,T\}$. Multiplying both sides of the inequalities \eqref{ineq: super diss ineq for sub sol of HJB} and \eqref{ineq: super BC}  by $-1$ we get $\forall x \in X_c, u \in Y, t\in [0,T] $
		\begin{align} \nonumber
	& \nabla_t (-J)(x,t) - c(x,u,t) + \nabla_x (-J)(x,t)^T f(x,u) \ge 0 \\ \nonumber	
	& -J(x,T) \le -g(x).
	\end{align}
	Using the previous part of the proof we deduce $-J$ is a sub solution to $\{-c,-g,f,X_0,Y,T\}$. Thus for any $t \in [0,T]$ and $x \in FR_f(X_0,Y,\{t\})$
		\small{\begin{align*}
	-J(x,t) \le \inf_{ \mbf u \in U_Y} \left\{ \int_{t}^{T} -c(\phi_f(x,s, \mbf u), \mbf u(s),s) ds - g(\phi_f(x,T, \mbf u)) \right\}.
	\end{align*} } \normalsize
	By multiplying both sides of the above inequality by $-1$ we deduce for any $t \in [0,T]$ and $x \in FR_f(X_0,Y,\{t\})$
			\begin{align} \nonumber
	& J(x,t) \\ \nonumber
	&\ge -\inf_{ \mbf u \in U_Y} \left\{ -\int_{t}^{T} c(\phi_f(x,s, \mbf u), \mbf u(s),s) ds + g(\phi_f(x,T, \mbf u)) \right\}\\ \nonumber
	& = \sup_{ \mbf u \in U_Y} \left\{ \int_{t}^{T} c(\phi_f(x,s, \mbf u), \mbf u(s),s) ds + g(\phi_f(x,T, \mbf u)) \right\}\\ \nonumber
	& \ge \inf_{ \mbf u \in U_Y} \left\{ \int_{t}^{T} c(\phi_f(x,s,u), \mbf u(s),s) ds + g(\phi_f(x,T, \mbf u)) \right\} \\ \label{proof super value }
	& = V^*(x,t).
	\end{align}
	Therefore it follows by \eqref{proof super value } that $J$ is a super-value function for $\{c,g,f,X_0,Y,T\}$.\end{proof}
%\begin{rem}
%	The Inequality \ref{ineq: diss ineq for sub sol of HJB} is equivalent to the form of dissipation inequalities:
%	\begin{equation*}
%	V(x,T) +
%	\end{equation*}
%\end{rem}
Next we give sufficient conditions for the existence of polynomial functions that satisfy Inequalities \eqref{ineq: diss ineq for sub sol of HJB}, \eqref{ineq: BC}, \eqref{ineq: super diss ineq for sub sol of HJB} and \eqref{ineq: super BC}. This proves the existence of polynomial sub-value and super-value functions but does not show that such functions can arbitrarily well approximate the true value function.

\begin{lem} \label{lem: existence of poly dissapating function}	
	For $T>0$; a compact set $Y \subset \R^m$; a compact set $X_0 \subseteq \R^n$; a polynomial function $g: \R^n \to \R$; a function $c \in C^{1}(\R^n \times \R^m \times \R)$; and $f \in \mcl F_Y$; suppose the set $FR_f(X_0,Y,[0,T])$ is bounded. Then there exists a polynomial sub-value function and polynomial super-value function to the optimal control problem $\{c,g,f,X_0,Y,T\}$.
\end{lem}
\begin{proof}
As $FR_f(X_0,Y,[0,T])$ is bounded it follows there exists $R>0$ such that $FR_f(X_0,Y,[0,T]) \subset B_R$. Now consider the polynomial function
		\begin{equation*}
		J_1(x,t)= g(x) + \underline{\alpha}(T-t),
		\end{equation*}
		where $\underline{\alpha}:=\inf_{x \in B_R, u \in Y, t \in[0,T]} \{ \nabla g(x) f(x,u) + c(x,u,t) \}$; which is well defined as the infimum of a differentiable function over a compact set is finite.
		
		To prove the existence of a polynomial sub-value function we show $J_1$ satisfies Inequalities \eqref{ineq: diss ineq for sub sol of HJB} and \eqref{ineq: BC}, and thus by Proposition \ref{prop: diss ineq implies lower soln} we deduce $J_1$ is a sub-value function for $\{c,g,f,X_0,Y,T\}$. Trivially \eqref{ineq: BC} holds. Now for $x \in FR_f(X_0,Y,[0,T])$ and $t \in [0,T]$
		\begin{align*}
		& \nabla_t J_1 (x,t) + c(x,u,t) + \nabla_x J_1(x,t)^T f(x,u) \\
		& = - \underline{\alpha} + c(x,u,t) + \nabla g(x)^T f(x,u) \\
		& \ge \inf_{x \in FR_f(X_0,Y,[0,T]), u \in Y, t \in[0,T]} \{c(x,u,t) + \nabla g(x)^T f(x,u)\} - \underline{\alpha}\\
		& \ge 0.
		\end{align*}
		Therefore we conclude $J_1$ satisfies \eqref{ineq: diss ineq for sub sol of HJB} and thus is a sub-value function to $\{c,g,f,X_0,Y,T\}$.
		
		 The existence of a super-value function follows by a similar argument and consideration of the function
		 \begin{align*}
		 J_1(x,t)= g(x) + \bar{\alpha}(T-t),
		 \end{align*}
		 where $\underline{\alpha}:=\sup_{x \in B_R, u \in Y, t \in[0,T]} \{ \nabla g(x) f(x,u) + c(x,u,t) \}$.
	\end{proof}

\section{Using SOS To Construct Sub-Value And Super-Value Functions} \label{sec: SOS to solve optimal control}
%For finite horizon optimal control problem $\{c,g,f,X_0,Y,T\}$ we have shown that if a function $V_l$ satisfies \eqref{ineq: BC} \eqref{ineq: diss ineq for sub sol of HJB} then it is a sub-value function for $\{c,g,f,X_0,Y,T\}$. Similarly we have shown that if a function $V_u$ satisfies \eqref{ineq: super BC} \eqref{ineq: super diss ineq for sub sol of HJB} then it is a super-value function to $\{c,g,f,X_0,Y,T\}$. We have also given sufficient conditions for a polynomial function to satisfy such dissipation inequalities in Lemma \ref{lem: existence of poly dissapating function}. We now propose an SOS optimization problem such that any feasible solution is a sub-value or super-value function.
For an optimal control problem $\{c,g,f,X_0,Y,T\}$ we would like to find the associated polynomial sub-value and super-value functions with minimum distance under some function metric; and hence are ``close" to a true value function. If we choose our function metric as the $L_1$ norm we seek to solve the optimization problem:
\begin{align*}
&\min_{V_u,V_l \in \mcl{P}_d[\R^n \times \R]} \left\{ \int_{\Omega} V_u(x,t_0)-V_l(x,t_0) dx  \right\}\\
& V_u(x,t) \ge V^*(x,t)\\
& V_l(x,t) \le V^*(x,t),
\end{align*}
where $V^*$ is a value function of $\{c,g,f,X_0,Y,T\}$. To enforce the constraints of the above optimization problem we use Proposition \ref{prop: diss ineq implies lower soln}; where it was shown if $V_l$ satisfies \eqref{ineq: BC} \eqref{ineq: diss ineq for sub sol of HJB} and $V_u$ satisfies \eqref{ineq: super BC} \eqref{ineq: super diss ineq for sub sol of HJB} then $V_l$ and $V_u$ are sub-value and super-value functions for $\{c,g,f,X_0,Y,T\}$ respectively. We then are able to tighten the optimization problem to an SOS optimization problem, indexed by $S(T,c,g,f,h_X,h_Y,d,\Omega)$:
\begin{align} \label{opt: SOS for sub soln of finite time}
&\min_{V_u,V_l \in \mcl{P}_d[\R^n \times \R]} \left\{ \int_{\Omega} V_u(x,t_0)-V_l(x,t_0) dx  \right\}\\ \nonumber
& \text{subject to: } k_{0,l},k_{1,l},k_{0,u},k_{1,u} \in \sum_{SOS}\\ \nonumber
&  s_{i,l},s_{i,u}\in \sum_{SOS} \text{ for } i=0,1,2,3
\end{align}
where
\begin{align*}
 k_{0,l}(x)& = (g(x)-V_l(x,T)) - s_{0,l}(x)h_X(x),\\
 k_{1,l}(x,u,t)& = \left( \nabla_t V_l(x,t) +c(x,u,t) + \nabla_x V_l(x,t)^T f(x,u)  \right) \\
&\quad -s_{1,l}(x,u,t)h_X(x) -s_{2,l}(x,u,t)h_Y(u) \\
& \qquad- s_{3,l}(x,u,t) (t)(T-t),\\
 k_{0,u}(x) & = (V_u(x,T) - g(x)) - s_{0,u}(x)h_X(x),\\
 k_{1,u}(x,u,t) & = -\left( \nabla_t V_l(x,t) +c(x,u,t) + \nabla_x V_u(x,t)^T f(x,u)  \right) \\
& \quad -s_{1,u}(x,u,t)h_X(x) -s_{2,u}(x,u,t)h_Y(u) \\
& \qquad - s_{3,u}(x,u,t) (t)(T-t).
\end{align*}

%In the next corollary we show how solutions to the above SOS optimization problem can be viewed as sub-value and super-value functions of a class of optimal control problems.
\begin{cor} \label{cor: SOS is solved by sub solution}
	Suppose $V_u$ and $V_l$ solve $S(T,c,g,f,h_X,h_Y,\Omega)$, given in \eqref{opt: SOS for sub soln of finite time}. Then $V_u$ and $V_l$ are super-value and sub-value functions to the optimal control problem $\{c,g,f,X_0,Y,T\}$ respectively; where $Y \subset \R^m$ and $X_0 \subset \R^n$ are such that $FR_f(X_0,Y,[0,T]) \subseteq \{x \in \R^n : h_X(x) \ge 0\}$ and $Y \subseteq \{u \in \R^m : h_Y(u) \ge 0\}$.
	
	Moreover if $\Omega \subseteq X_0$ the following holds,
	\begin{align} \label{error bounds}
	||V^*(\cdot,0) - V_l(\cdot,0)||_{1, \Omega} \le \eps \text{ and } ||V^*(\cdot,0) - V_u(\cdot,0)||_{1, \Omega} \le \eps,
	\end{align}
	where $\eps=\int_{\Omega} V_u(x,0)-V_l(x,0) dx$ and $V^*$ is the value function of the optimal control problem $\{c,g,f,X_0,Y,T\}$.
\end{cor}
\begin{proof}
	We first prove $V_l$ is a sub-value function by showing $V_l$ satisfies the dissipation inequalities \eqref{ineq: diss ineq for sub sol of HJB} and \eqref{ineq: BC}; as it follows by Proposition \ref{prop: diss ineq implies lower soln} that such a function must be a sub-value function of $\{c,g,f,X_0,Y,T\}$.
	
As $k_{0,l} \in \sum_{SOS}$ it follows $k_{0,l}(x) \ge 0$ for all $x \in \R^n$. Moreover since a positive function multiplied by a positive function is a postive function we furthermore deduce
\begin{align*}
V_l(x,T) \le g(x) \quad \forall x \in \{y \in \R^n: h_X(y) \ge 0\}.
\end{align*}
As $FR_f(X_0,Y,[0,T]) \subseteq \{x \in \R^n : h_X(x) \ge 0\}$ the above inequality also holds for all $x \in FR_f(X_0,Y,[0,T])$. Therefore $V_l$ satisfies Inequality \eqref{ineq: BC}.

As $k_{1,l} \in \sum_{SOS}$ it follows for all $x \in \{y \in \R^n: h_X(y) \ge 0\}$, $u \in \{w \in \R^m: h_Y(w) \ge 0\}$, and $t \in \{s \in \R: [s][T-s] \ge 0\}$
\begin{align*}
\nabla_t V_l(x,t) +c(x,u,t) + \nabla V_l(x,t)^T f(x,u) \ge 0.
\end{align*}
As $FR_f(X_0,Y,[0,T]) \subseteq \{x \in \R^n : h_X(x) \ge 0\}$, $Y \subseteq \{w \in \R^m: h_Y(w) \ge 0\}$ and $[0,T]=\{s \in \R: [s][T-s] \ge 0\}$ it follows $V_l$ satisfies Inequality \eqref{ineq: diss ineq for sub sol of HJB}. Therefore we conclude $V_l$ is a sub-value function as it satisfies the Inequality \eqref{ineq: diss ineq for sub sol of HJB} and \eqref{ineq: BC}. Moreover, it follows by a similar argument to the above that $V_u$ is a super-value function.

Finally the error bounds in \eqref{error bounds} immediately follows using $V_l(x,0) \le V^*(x,0) \le V_u(x,0)$ for all $x \in \Omega \subseteq X_0$ and $t \in [0,T]$. \end{proof}

In Lemma \ref{lem: under and over approx of reachable set} we saw how sub-value and super-value functions over- and inner-bound reachable sets. In the next corollary we will show how solutions to the SOS Optimization Problem \eqref{opt: SOS for sub soln of finite time} also over and inner bound reachable sets.

\begin{cor} \label{cor: SOS for reachable set estimation}
	Suppose $V_u$ and $V_l$ solve $S(T,0,g,f,h_X,h_Y,\Omega)$, given in \eqref{opt: SOS for sub soln of finite time}. Let $Y = \{u \in \R^m : h_Y(u) \ge 0\}$ and $X_0=\{x \in \R^n: g(x) \le 1 \}$. Suppose for some $X \subset \R^n$ such that $BR_{f}(X_0,Y,\{T\}) \subseteq X$ the following holds $FR_f(X,Y,[0,T]) \subseteq \{x \in \R^n : h_X(x) \ge 0\}$. Then
		\begin{align} \label{sos set containments}
	\{x \in X: V_u(x,0) \le 1\} &  \subseteq BR_{f}(X_0,Y,\{T\})\\ \nonumber
	BR_{f}(X_0,Y,\{T\}) & \subseteq \{x \in X : V_l(x,0) \le 1\}.
	\end{align}
\end{cor}
\begin{proof}
	By Corollary \ref{cor: SOS is solved by sub solution} the functions $V_u$ and $V_l$ are super-value and sub-value functions to the optimal control problem $\{0,g,f,X,Y,T\}$ where $BR_{f}(X_0,Y,\{T\}) \subseteq X$. Therefore by Lemma \ref{lem: under and over approx of reachable set} the set containments \eqref{sos set containments} hold.
	\end{proof}
For reachable set analysis using $S(T,0,g,f,h_X,h_Y,d,\Omega)$, given in \eqref{opt: SOS for sub soln of finite time}, typically we select $h_X= R^2 - x_1^2 -x_2^2$ for $R>0$ so $\{x \in \R^n : h_X(x) \ge 0\} = B_R$. Then, assuming the set $BR_f(X_0,f,{T})$ is compact, we select $R>0$ sufficiently large enough for there to exist a compact set $X \subset \R^n$ such that $BR_f(X_0,f,{T}) \subseteq X$ and $FR_f(X,Y,[0,T]) \subseteq B_R$. Knowledge of the set $X \subset \R^n$ is not necessary to construct an outer approximation of the backward reachable set; as by Corollary \ref{cor: SOS for reachable set estimation} we have $BR_{f}(X_0,Y,\{T\}) \subseteq \{x \in X : V_l(x,0) \le 1\} \subseteq \{x \in \R^n: V_l(x,0) \le 1\}$, where $(V_u,V_l)$ solve $S(T,0,g,f,h_X,h_Y,d,\Omega)$.

\subsection{Numerical Example: Using SOS To Numerically Approximating A Non-Differentiable Value Function}

\begin{figure}
%		\vspace{-5pt}
	\includegraphics[scale=0.6]{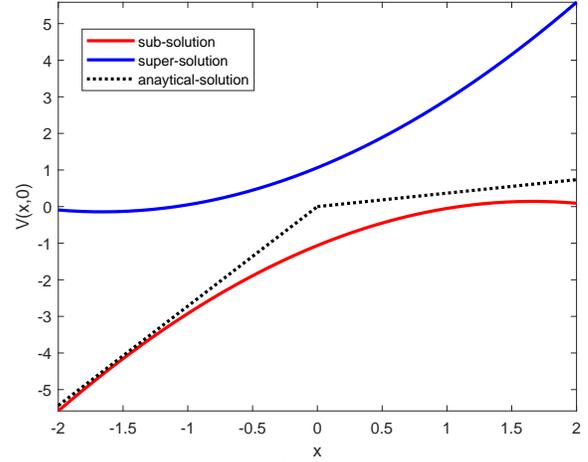}
	\vspace{-25pt}
	\caption{The value function $V(x,t)$, given in \eqref{analytucal soln}, plotted as the dotted black line, along with the sub-value function, plotted as the red line, and super-value function, plotted as the blue line, found by solving the SOS Optimization Problem \eqref{opt: SOS for sub soln of finite time}.}
	\label{fig: finite time sos}
	\vspace{-15pt}
\end{figure}

Let $X_0=[-8,8]$; $T>0$; $Y=[-1,1]$; $c(x,t)=0$ for all $x \in \R$ and $t>0$; $g(x)=x$; $f(x,u)=xu$ and consider the optimal control problem $\{c,g,f,X_0,Y,T\}$. It was shown in \cite{liberzon2011calculus} that the value function of $\{c,g,f,X_0,Y,T\}$ can be analytically found as
	\begin{equation} \label{analytucal soln}
	V(x,t)= \begin{cases} \exp(t-T)x \text{ if } x>0,\\
	\exp(T-t)x \text{ if } x<0,\\
	0 \text{ if } x=0. \end{cases}
	\end{equation}
	We note that $V$ is not differentiable at $x=0$ but can be shown to satisfy the associated HJB PDE away from $x=0$. This problem shows how the value function can be non-smooth even for simple optimal control problems with polynomial vector field and cost. We next attempt to find a polynomial, and thus smooth, super-value and sub-value functions of this optimal control problem that is close to the non-smooth value function given in \eqref{analytucal soln} under the $L_1$ norm.
	
	We numerically solved the SOS optimization problem $S(T,c,g,f,h_X,h_Y,d,\Omega)$ with $T=1$; $c$, $g$ and $f$ the same as the above optimal control problem; $h_X=8^2 - x^2$; $h_Y(u)=(-1-u)(u-1)$; $d=4$; $\Omega=[-2,2]$. The result is displayed in Figure \ref{fig: finite time sos} where the exact value function, given in \eqref{analytucal soln}, is plotted as the dotted line and super-value and sub-value functions are plotted as the blue and red line respectively. We see even though the exact value function is discontinuous at $x=0$ the smooth polynomial sub-value is a reasonable tight approximation.

\subsection{Numerical Examples: Using SOS To Solve The HJB PDE For Reachable Set Approximation}
%In this section we would like to compute the forward reachable sets of several ODE's. To do this we use Corollary \ref{cor: SOS for reachable set estimation} that shows how the SOS optimization problem, given in \eqref{opt: SOS for sub soln of finite time}, is able to construct inner and over bounds on the backward reachable set. Then using Lemma \ref{lem: backward and forward reach sets}  we are able to deduce what the forward reachable set is.

\begin{figure}
%	\vspace{-5pt}
	\includegraphics[scale=0.6]{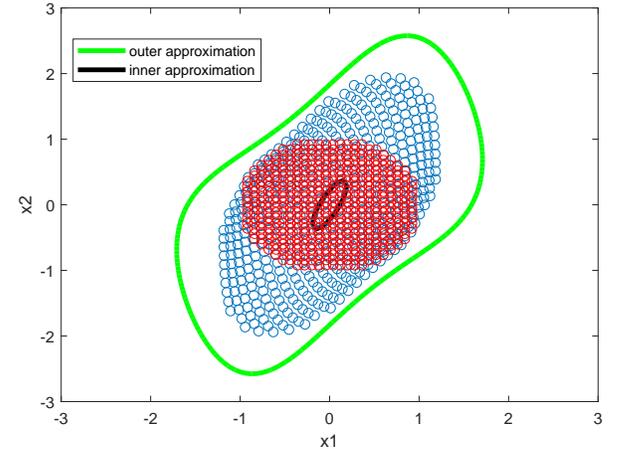}
			\vspace{-20pt}
	\caption{The 1 sublevel set at $t=0$ of the sub-value function, the green curve, and super-value function, the black curve, found by solving the SOS optimization problem \eqref{opt: SOS for sub soln of finite time} for the ODE \eqref{eqn: van der pol ode}. The blue points show the solution map of the ODE \eqref{eqn: van der pol ode} at $T=1$ starting from initial conditions shown as the red points.}
	\label{fig: van }
		\vspace{-20pt}
\end{figure}

\begin{ex}
	Let us now consider the Van der Pol oscillator defined by the nonlinear ODE:
	\begin{align} \label{eqn: van der pol ode}
	\dot{x}_1(t) & = x_2(t)\\ \nonumber
	\dot{x}_2(t) & = -x_1(t) +  x_2(t)(1- x_1^2(t)),
	\end{align}
	
%	where $ \mu \in U=\{u \in L_2^{1}[0,T]: u(t) \in [\underline{u}, \bar{u} ] \text{ for all } t \in [0,T] \}$ is a modeling parameter that measures damping strength.
	
	To find the forward reachable set for the Van der Pol oscillator we solved the optimization problem $S(T,c,g,f,h_X,h_Y,d,\Omega)$, found in \eqref{opt: SOS for sub soln of finite time}, with $T=1$; $c=0$; $g(x)=x_1^2 + x_2^2$; $f(x)= -[x_2, -x_1 +  x_2(1-x_1^2)]^T$; $h_X(x)= 10^2 - x_1^2 - x_2^2$; $h_Y(u)=0$; $d=4$ and $\Omega=[-2,2]\times [-2,2]$. The sublevel sets $\{x \in \R^n: V_u(x,0) \le 1\}$ and $\{x \in \R^n: V_l(x,0) \le 1\}$, where $(V_u,V_l)$ solve the above optimization problem, are then plotted in Figure \ref{fig: van } as the black line and green line respectively. As shown in Corollary \ref{cor: SOS for reachable set estimation} these sublevel sets are over and under set approximations of $BR_{f}(X_0,Y,\{T\})$, which was shown to be equal to $FR_{-f}(X_0,Y,\{T\})$ in Lemma \ref{lem: backward and forward reach sets}, where $X_0=\{x \in \R^n: g(x) \le 1\}$. This is clearly demonstrated in Figure \ref{fig: van } where the red points represent initial points contained inside the set $X_0$ and blue points represent points the solution map can transition to at time $T=1$ starting in $X_0$; where both sets of points were approximately found from forward time integrating \eqref{eqn: van der pol ode}.
\end{ex}

\begin{ex}
	Let us consider the linear ODE:
	\begin{equation} \label{ODE: linear}
	\dot{x}(t) = \mbf u(t)Ax(t),
	\end{equation}
	where $A = \begin{bmatrix}
	0 &-1 \\ 1 & 0
	\end{bmatrix}$. Since the eigenvalues of $A$ are $ \pm i$ it follows \eqref{ODE: linear} produces non-stable circular trajectories for fixed input $\mbf u(t) \equiv u \in \R^m$.
	
	To find the forward reachable set for this linear ODE \eqref{ODE: linear} for fixed input $\mbf u(t)\equiv1$ we solved the optimization problem $S(T,c,g,f,h_X,h_Y,d,\Omega)$, found in \eqref{opt: SOS for sub soln of finite time}, for both $d=4$ and $d=4$ with $T=5$; $c=0$; $g(x)=(x_1-1.5)^2 + x_2^2$; $f(x)= -Ax$; $h_X(x)= 10^2 - x_1^2 - x_2^2$; $h_Y(u)=0$; and $\Omega=[-3,3]\times [-3,3]$. We plotted the 1-sublevel sets at time $0$ of the solutions to these optimization problem, $V_u$ and $V_l$, in Figure \ref{fig: linear d=3} as the black line and green line respectively; where the dotted lines are for $d=3$ and filled lines for $d=4$. Here the red points represent initial points contained inside the set $X_0=\{x \in \R^n: g(x) \le 1\}$ and blue points represent points the solution map can transition to at time $T=5$ starting in $X_0$; where both sets of points were approximately found from forward time integrating \eqref{ODE: linear}. As expected, by Corollary \ref{cor: SOS for reachable set estimation}, we see these sublevel sets under and over approximate the reachable set respectively. We also see increasing the degree makes our approximations tighter.
	
	We have furthermore approximated the forward reachable set of the linear ODE \eqref{ODE: linear} when the input is allowed to vary but constrained inside the set $Y=[-2,2]$. To do this we solved the optimization problem $S(T,c,g,f,h_X,h_Y,d,\Omega)$, found in \eqref{opt: SOS for sub soln of finite time}, with $T=0.5$; $c=0$; $g(x)=(x_1-1.5)^2 + x_2^2$; $f(x)= -Ax$; $h_X(x)= 4^2 - x_1^2 - x_2^2$; $h_Y(u)=(u+2)(2-u)$; $d=2$ and $\Omega=[-3,3]\times [-3,3]$. In Figure \ref{fig: input } we then plotted $\{x \in \R^2: V_l(x,0) \le 1\}$ as the green line, where $(V_u,V_l)$ solves the above optimization problem. By Corollary \ref{cor: SOS for reachable set estimation} the set  $\{x \in \R^2: V_l(x,0) \le 1\}$ over approximates the set $BR_{f}(X_0,Y,\{T\})$, shown in Lemma \ref{lem: backward and forward reach sets} to be equal to $FR_{-f}(X_0,Y,\{T\})$, where $X_0=\{x \in \R^n: g(x) \le 1\}$. This is demonstrated in Figure \ref{fig: input } as the terminal points of the solution map at time $T=0.5$, represented by the blue points, are all contained inside the green line.
	
	\begin{figure}
	%		\vspace{-5pt}
		\includegraphics[scale=0.6]{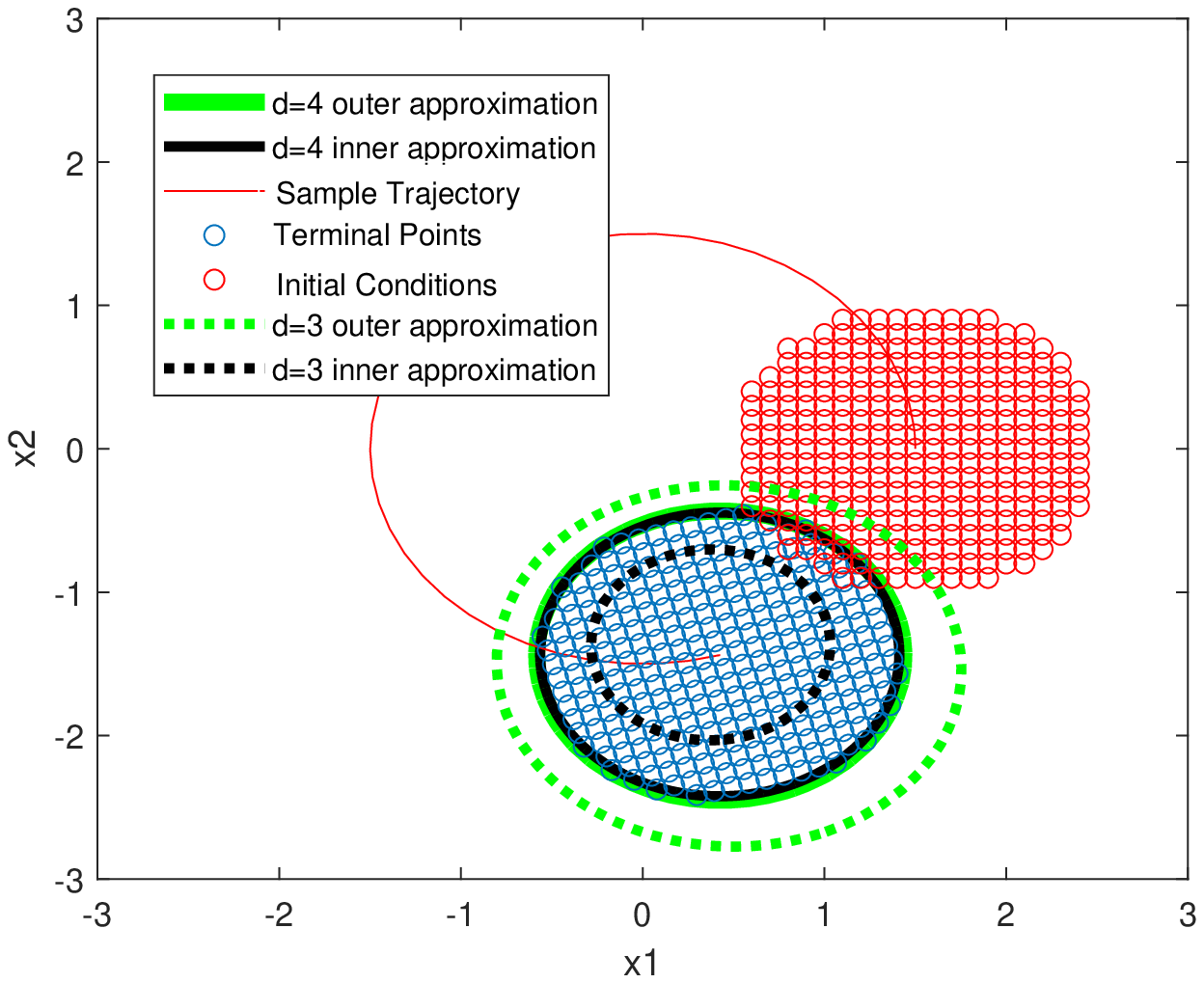}
		\vspace{-20pt}
		\caption{The 1 sublevel set at $t=0$ of the sub-value function, the green curve, and super-value function, the black curve, found by solving the SOS optimization problem \eqref{opt: SOS for sub soln of finite time} for the ODE \eqref{ODE: linear} with $\mbf u(t)\equiv1$ for $d=3,4$. The blue points show the solution map of the ODE \eqref{ODE: linear} at $T=5$ starting from initial conditions, shown as red points.}
		\label{fig: linear d=3}
			\vspace{-20pt}
	\end{figure}
	
	\begin{figure}
%			\vspace{-5pt}
		\includegraphics[scale=0.6]{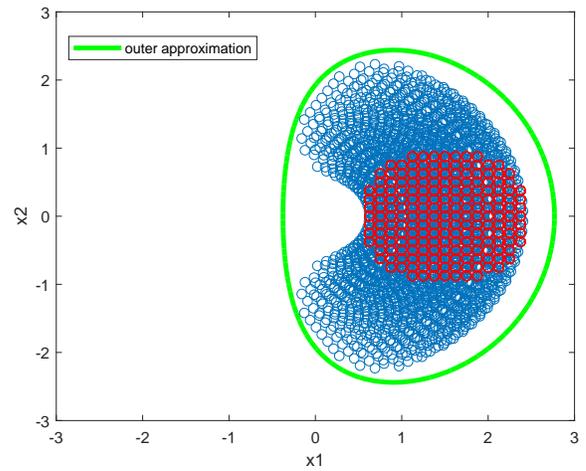}
		\vspace{-20pt}
		\caption{The 1 sublevel set at $t=0$ of the sub-value function, the green curve, found by solving the SOS optimization problem \eqref{opt: SOS for sub soln of finite time} for the ODE \eqref{ODE: linear} with $\mbf u \in U_Y$ and $Y=[-2,2]$. The blue points show the solution map of the ODE \eqref{ODE: linear} at $T=0.5$ starting from initial conditions, shown as red points, for various inputs $\mbf u \in U_Y$.}
		\label{fig: input }
			\vspace{-15pt}
	\end{figure}
\end{ex}

\section{Conclusion} \label{sec: conclusion}
In this paper we have shown if a function satisfies dissipation inequalities then it is a sub-value or super-value function to an optimal control problem. Further to this we have given sufficient conditions for the existence of polynomial sub-value and super-value functions to optimal control problems. An SOS optimization problem was proposed that is solved by sub-value and super-value functions of an optimal control problem that have minimum $L_1$ norm. It was shown how this SOS optimization problem is able to construct outer and inner set approximations of reachable sets.

\section*{Acknowledgements}
This work was supported by the National Science Foundation under grants No. 1538374 and 1739990.

\bibliographystyle{unsrt}
\bibliography{Jones_CDC_2019_v4}

\begin{thebibliography}{10}

\bibitem{kampmeier2018reaction}
Jennifer Kampmeier, Reidar Larsen, Lucas~F Migliorini, and Kipp~A Larson.
\newblock Reaction wheel performance characterization using the kepler
  spacecraft as a case study.
\newblock In {\em 2018 SpaceOps Conference}, page 2563, 2018.

\bibitem{chen2018study}
Zhou Chen, De-Yuan Deng, Quan-Sheng Yan, Jin-Zhong Lu, and Jian-Xin Lu.
\newblock Study on nonlinear lateral parameter bifurcation characteristic of
  soft footbridge.
\newblock In {\em IOP Conference Series: Materials Science and Engineering},
  volume 322, page 042036. IOP Publishing, 2018.

\bibitem{eckhardt2007modeling}
Bruno Eckhardt, Edward Ott, Steven~H Strogatz, Daniel~M Abrams, and Allan
  McRobie.
\newblock Modeling walker synchronization on the millennium bridge.
\newblock {\em Physical Review E}, 75(2):021110, 2007.

\bibitem{mitchell2005time}
Ian~M Mitchell, Alexandre~M Bayen, and Claire~J Tomlin.
\newblock A time-dependent hamilton-jacobi formulation of reachable sets for
  continuous dynamic games.
\newblock {\em IEEE Transactions on automatic control}, 50(7):947--957, 2005.

\bibitem{bertsekas2005dynamic}
Dimitri~P Bertsekas.
\newblock {\em Dynamic programming and optimal control}, volume~1.
\newblock Athena scientific Belmont, MA, 2005.

\bibitem{crandall1992user}
Michael~G Crandall, Hitoshi Ishii, and Pierre-Louis Lions.
\newblock User’s guide to viscosity solutions of second order partial
  differential equations.
\newblock {\em Bulletin of the American mathematical society}, 27(1):1--67,
  1992.

\bibitem{wang2017adaptive}
Zhong Wang and Yan Li.
\newblock An adaptive cross approximation method for the
  hamilton-jacobi-bellman equation.
\newblock {\em IFAC-PapersOnLine}, 50(1):6289--6294, 2017.

\bibitem{wan2019alternating}
Changhuang Wan, Ran Dai, and Ping Lu.
\newblock Alternating minimization algorithm for polynomial optimal control
  problems.
\newblock {\em Journal of Guidance, Control, and Dynamics}, pages 1--14, 2019.

\bibitem{summers2013approximate}
Tyler~H Summers, Konstantin Kunz, Nikolaos Kariotoglou, Maryam Kamgarpour, Sean
  Summers, and John Lygeros.
\newblock Approximate dynamic programming via sum of squares programming.
\newblock In {\em 2013 European Control Conference (ECC)}, pages 191--197.
  IEEE, 2013.

\bibitem{leong2016linearly}
Yoke~Peng Leong, Matanya~B Horowitz, and Joel~W Burdick.
\newblock Linearly solvable stochastic control lyapunov functions.
\newblock {\em SIAM Journal on Control and Optimization}, 54(6):3106--3125,
  2016.

\bibitem{greenstreet1999reachability}
Mark~R Greenstreet and Ian Mitchell.
\newblock Reachability analysis using polygonal projections.
\newblock In {\em International Workshop on Hybrid Systems: Computation and
  Control}, pages 103--116. Springer, 1999.

\bibitem{maidens2015reachability}
John Maidens and Murat Arcak.
\newblock Reachability analysis of nonlinear systems using matrix measures.
\newblock {\em IEEE Transactions on Automatic Control}, 60(1):265--270, 2015.

\bibitem{summers2013quantitative}
Erin Summers, Abhijit Chakraborty, Weehong Tan, Ufuk Topcu, Pete Seiler, Gary
  Balas, and Andrew Packard.
\newblock Quantitative local l2-gain and reachability analysis for nonlinear
  systems.
\newblock {\em International Journal of Robust and Nonlinear Control},
  23(10):1115--1135, 2013.

\bibitem{yin2018reachability}
He~Yin, Andrew Packard, Murat Arcak, and Peter Seiler.
\newblock Reachability analysis using dissipation inequalities for nonlinear
  dynamical systems.
\newblock {\em arXiv preprint arXiv:1808.02585}, 2018.

\bibitem{yin2018finite}
He~Yin, Andrew Packard, Murat Arcak, and Pete Seiler.
\newblock Finite horizon backward reachability analysis and control synthesis
  for uncertain nonlinear systems.
\newblock {\em arXiv preprint arXiv:1810.00313}, 2018.

\bibitem{jones2019using}
Morgan Jones and Matthew~M Peet.
\newblock Using sos and sublevel set volume minimization for estimation of
  forward reachable sets.
\newblock {\em arXiv preprint arXiv:1901.11174}, 2019.

\bibitem{liberzon2011calculus}
Daniel Liberzon.
\newblock {\em Calculus of variations and optimal control theory: a concise
  introduction}.
\newblock Princeton University Press, 2011.

\end{thebibliography}

\end{document}